\documentclass[12pt]{article}
\usepackage{amssymb,amsmath,latexsym}

\usepackage{amssymb}
\usepackage{amsmath}
 \usepackage{amsthm}
\numberwithin{equation}{section}
\newtheorem{thm}{\bf Theorem}[section]

\newtheorem{prop}[thm]{\bf Proposition}
\newtheorem{cor}[thm]{\bf Corollary}
\newtheorem{defn}{\bf Definition}[section]
\theoremstyle{remark}
\newtheorem{rem}{\bf Remark}[section]
\newtheorem{exmp}{\bf Example}[section]

\begin{document}
\title{A note on power instability of linear discrete-time systems  in Banach
spaces}
\author{\large{\bf Ioan-Lucian Popa}\\
Department of Mathematics, Faculty of Mathematics and Computer Science, \\
West University of Timi\c soara,  Romania\\
 e-mail: popa@math.uvt.ro}

\date{ }

\maketitle

\begin{abstract}
In this paper we investigate the power instability properties and
give necessary and sufficient conditions for the concepts of
uniform power instability, power instability  and strong power
instability for linear discrete-time system $x_{n+1}=A{(n)}x_{n}$
in Banach spaces.
\end{abstract}

{\it Keywords:} linear discrete-time systems, uniform power
instability, (strong) power instability
\section{Introduction and Preliminaries}
%
%
%
%
%

\hspace{1cm}The qualitative theory of difference equations is in a
process of
 continuous development in the past decades (see, e.g.
 \cite{halanay}, \cite{agarwal}, \cite{Elaydi1}, \cite{mil} and the references therein).
An important result in the theory of linear discrete-time systems
have been proved by R.K. Przyluski and S. Rolewicz in
\cite{przyluski} for the concept of uniform power stability.
In our previous paper \cite{popa} we obtained characterizations
for uniform an nonuniform exponential stability concepts for
linear discrete-time systems. Other results, concerning nonuniform
exponential stability concepts have been studied by L. Barreira
and C. Valls in \cite{bareira2}. Diverse and important concepts of
instability have been introduced and studied (see \cite{Megan2}
for semigroups of operators, \cite{Megan3}, \cite{tomescu} for
evolution operators, \cite{Megan1} for linear skew-product flows).

\hspace{1cm}In this paper we present the concept  of uniform power
instability and two nonuniform concepts (power instability and
strong power instability) for linear discrete-time systems. Our
main objectives are to establish relations between these concepts
and to offer generalizations of R.K. Przyluski type theorem for
these concepts.

\hspace{1cm} Let us first introduce the notation used in this
note. Let $X$ be a real or complex Banach space and
$\mathcal{B}(X)$ the Banach algebra of all bounded and linear
operators from $X$ into itself. The norm on $X$ and in
$\mathcal{B}(X)$  will be denoted by $\parallel.\parallel.$  The
set of all positive integers will be denoted by $\mathbb{N},$
$\Delta$ denotes the set of all pairs $(m,n)$ of positive integers
satisfying the inequality $m\geq n.$ We also denote by $T$ the set
of all triplets $(m,n,p)$ of positive integers with $(m,n)$ and
$(n,p)\in\Delta.$

We consider the linear discrete-time system:

\begin{equation*}\tag{$\mathfrak{A}$}\label{A}
x_{n+1}=A{(n)}x_{n},
\end{equation*}
where $A:\mathbb{N}\rightarrow \mathcal{B}(X)$ is a given
$\mathcal{B}(X)-$valued sequence. For $(m,n)\in\Delta$ we denote:
\begin{equation}\label{eqAmn}
\mathcal{A}_{m}^{n}=\left\{\begin{array}{ll}
A{(m)}\cdot \ldots\cdot  A{(n+1)},\;\; \;\; m > n\\
\qquad\quad I\qquad\qquad\quad,\;\;\;\; m=n.
\end{array}\right.
\end{equation}

\begin{defn}\label{def1}
The linear discrete-time system  (\ref{A}) is said to be
 {\it uniformly power instable} (and denote u.p.is.) if there are some
 constants $N\geq 1$ and $r\in (0,1)$ such that:
 \begin{equation*}\label{ueis}
 \parallel \mathcal{A}_{n}^{p}x\parallel \leq N r^{m-n}\parallel
 \mathcal{A}_{m}^{p} x\parallel,\;\; \text{for all}\;\; (m,n,p,x)\in T\times X.
 \end{equation*}
\end{defn}
\begin{defn}\label{def2}
The linear discrete-time system (\ref{A}) is said to be:
 \begin{description}

\item[{\it{i)}}] {\it  power instable} (and denote p.is.) if there
are some constants $N \geq 1,$ $r\in(0,1)$ and $s\geq 1$ such
that:
\begin{equation}\label{eis}
\parallel \mathcal{A}_{n}^{p}x\parallel \leq N r^{m-n}s^{ n}
\parallel \mathcal{A}_{m}^{p}x\parallel,\;\;  \text{for all}\;\; (m,n,p,x)\in T\times X.
\end{equation}

\item[{\it{ii)}}] {\it strongly power instable} (and denote
s.p.is.) if there are some constants $N \geq 1,$ $r\in (0,1)$ and
$s\in \left[1, \dfrac{1}{r}\right)$ such that:
\begin{equation}\label{ses}
\parallel \mathcal{A}_{n}^{p}x\parallel \leq N r^{m-n}s^{n} \parallel
\mathcal{A}_{m}^{p}x\parallel,\;\;  \text{for all}\;\;
(m,n,p,x)\in T\times X.
\end{equation}
\end{description}
\end{defn}
\section{Main results}

From the previous definitions it follows that ($u.p.is
\Longrightarrow s.p.is. \Longrightarrow p.is. $) The next example
illustrate the difference between the concepts of
 (strong) power instability and  uniform power instability for
linear discrete-time system (\ref{A}). We remark that the concepts
of power instability and strong power instability are much weaker
behaviors in comparison with the classical concept of uniform
power instability. A principal motivation for weakening the
assumption is that almost all variational equations in a finite
dimensional spaces have a nonuniform exponential behavior.
\begin{exmp}
Let $X=\mathbb{R}$ be a Banach space, $c>0,$ $(a_n)_{n}\subset
(0,\infty)$  and $(A_n)_{n}\subset\mathcal{B(\mathrm{X})}$ defined
for all $n\in\mathbb{N}$ by $A_{n} =c\cdot a_{n}I ,$ where

\[
  a_{n} = \left\{
  \begin{array}{l l}
    2^{-n} & \quad \text{if $n=2k$}\\
    2^{n+1} & \quad \text{if $n=2k+1$}\\
  \end{array} \right.
\]
The following statements are true:
\begin{description}

\item[{\it{i)}}](\ref{A}) is not uniformly power instable;

\item[{\it{ii)}}](\ref{A}) is power instable if and only if $c
>1$;

\item[{\it{iii)}}] (\ref{A}) is strongly power instable if and
only if $c>e.$
\end{description}

Let $(m,n,x)\in\Delta\times X.$ According to (\ref{eqAmn}) we have
that:
\[
  \mathcal{A}_{m}^{n}x = \left\{
  \begin{array}{l l}
    c^{m-n}a_{mn}x & \quad m>n\\
    x & \quad m=n\\
  \end{array} \right. ,
\]
 where
\[
  a_{mn} = \left\{
  \begin{array}{l l}
    1 & \quad \text{if $m=2q$ and $n=2p$}\\
    2^{-n-1} & \quad \text{if $m=2q$ and $n=2p+1$}\\
    2^{m+1} & \quad \text{if $m=2q+1$ and $n=2p$}\\
    2^{m-n} & \quad \text{if $m=2q+1$ and $n=2p+1$}\\
  \end{array} \right.
\]

$(i)$ If we suppose that (\ref{A}) is u.p.is. then there exist
some constants $N \geq 1$ and $r\in(0,1)$ such that
$$\parallel
x\parallel \leq N r^{m-n} \parallel \mathcal{A}_{m}^{n}x\parallel=
N (rc)^{m-n}a_{mn}\parallel x\parallel,$$ for all
$(m,n,x)\in\Delta\times X$ which  is equivalent with
\[
  \left\{
  \begin{array}{l l}
    \left( \dfrac{1}{rc}\right)^{m-n} \leq N & \quad \text{if $m=2q$ and $n=2p$}\\
    \left( \dfrac{1}{rc}\right)^{m-n}2^{n+1} \leq N & \quad \text{if $m=2q$ and $n=2p+1$}\\
    \left( \dfrac{1}{rc}\right)^{m-n}2^{-m-1} \leq N & \quad \text{if $m=2q+1$ and $n=2p$}\\
    \left( \dfrac{1}{rc}\right)^{m-n}2^{-m+n} \leq N & \quad \text{if $m=2q+1$ and $n=2p+1$}.\\
  \end{array} \right.
\]
There are two cases that can be considered at this point. If $c\in
(0,1]$ then for all $r\in(0,1),$ $m=2q$ and $n=2p\in \mathbb{N}$
fixed we have that
\begin{equation}\label{ex1eq1}
\lim\limits_{q\rightarrow\infty} \left(
\frac{1}{rc}\right)^{2q-2p}=\infty.
\end{equation}
If $c\in (1,\infty)$, $r\in (0,1),$ $n=2p+1$ and $m=n+1$ it
follows that
\begin{equation}\label{ex1eq2}
\lim\limits_{p\rightarrow\infty} \left(
\frac{1}{rc}\right)2^{2p+2}=\infty.
\end{equation}
According to (\ref{ex1eq1}) and (\ref{ex1eq2}) we can conclude
that (\ref{A}) can not be u.p.is.

$(ii)$ If (\ref{A}) is power instable then there exist some
constants $N \geq 1,$ $r\in(0,1)$ and $s \geq 1$ such that:
$$\parallel x\parallel \leq N r^{m-n}s^{n} \parallel
\mathcal{A}_{m}^{n}x\parallel= N (rc)^{m-n}s^{n}a_{mn}\parallel
x\parallel,$$ for all $(m,n,x)\in\Delta\times X$ which  is
equivalent with
\begin{equation}\label{ex1eq3}
  \left\{
  \begin{array}{l l}
    \left( \dfrac{1}{rc}\right)^{m-n}{s}^{-n} \leq N & \quad \text{if $m=2q$ and $n=2p$}\\
    \left( \dfrac{1}{rc}\right)^{m-n}s^{-n}2^{n+1} \leq N & \quad \text{if $m=2q$ and $n=2p+1$}\\
    \left( \dfrac{1}{rc}\right)^{m-n}s^{-n}2^{-m-1} \leq N & \quad \text{if $m=2q+1$ and $n=2p$}\\
    \left( \dfrac{1}{rc}\right)^{m-n}s^{-n}2^{-m+n} \leq N & \quad \text{if $m=2q+1$ and $n=2p+1$}.\\
  \end{array} \right.
\end{equation}
In this case, for any $q\in\mathbb{N}^{*}$ we have that
$(2q,2q-1)\in\Delta,$  $N\geq \dfrac{2}{rc}
\left(\dfrac{2}{s}\right)^{2q-1},$ and obvious $s > 2.$ Since
relation (\ref{ex1eq3})  is true we have that $ 0< \dfrac{1}{rc}
\leq 1.$ If $\dfrac{1}{rc} >1,$ then for  $m=2q$ and $n=2$ we
obtain:
\begin{equation}\label{ex1eq4}
\lim\limits_{q\rightarrow\infty} \left(
\frac{1}{rc}\right)^{2q-2}\left(\frac{1}{s}\right)^{2}=\infty.
\end{equation}

Reciprocally, we suppose that $c >1.$  In this case we consider $N
\geq e,$ $r \in \left(0,\frac{1}{c}\right]$ and $s>2$ and in this
way (\ref{ex1eq3})  is verified for all $(m,n)\in\Delta.$ More
than that if $c\in (0,1]$ then for all $r\in(0,1),$ $s>1$ and
$n=2$ we obtain (\ref{ex1eq4}). Hence, for $c>1,$ (\ref{A}) is
p.is.

$(iii)$ If we suppose that (\ref{A}) is s.p.is. then there exist
some constants $N\geq 1,$ $p\in(0,1)$ and $s\in
\left[1,\dfrac{1}{r}\right)$ such that for all $(m,n)\in\Delta$
relation (\ref{ex1eq3}) to be true. In a similar way as we proved
$(ii)$ we obtain that $c\geq \dfrac{1}{r}>s\geq e.$
 If $ 0<c\leq e,$ with $\dfrac{1}{r}>s\leq e,$ then for $m=2q$ and
 $n =2$ we obtain (\ref{ex1eq4}).

Reciprocally, we consider  $c > e.$ In this case we consider $N
\geq e,$ $e>r\geq\dfrac{1}{c}$ and $s=2$ and in this way
(\ref{ex1eq3}) is verified for all $(m,n)\in\Delta.$ Hence, for
$c> e,$  (\ref{A}) is s.p.is.
\end{exmp}
\begin{thm}\label{theorem2}
The linear discrete-time system (\ref{A}) is power instable if and
only if there exist some constants $D \geq 1,$ $d
>1$ and $c \geq 1$ with $c\in [1,d)$ such that:
\begin{equation}\label{t26eqii}
\sum\limits_{k=n}^{m} d^{m-k}\parallel \mathcal{A}_{k}^{n}
x\parallel \leq D c^{m}
\parallel \mathcal{A}_{m}^{n}x\parallel,\;\; \text{for all} \;\;(m,n,x)\in \Delta \times X.
\end{equation}
\end{thm}
\begin{proof}
{\it Necessity.} According to Definition \ref{def2} there are some
constants $N \geq 1,$ $r\in (0,1)$ and $s \geq 1$ such that:
\begin{equation*}
\parallel \mathcal{A}_{k}^{n}x\parallel \leq N r^{m-k}s^{k}
\parallel \mathcal{A}_{m}^{n}x\parallel,\;\;\text{for all}\;\; (m,k,n,x)\in T\times X.
\end{equation*}
Then for any $d>\dfrac{s}{r}$ and all $(n,x)\in
 \mathbb{N}\times X$ we have:
\begin{align*}
\sum\limits_{k=n}^{m} d^{m-k}\parallel \mathcal{A}_{k}^{n}
x\parallel &\leq N
\parallel \mathcal{A}_{m}^{n}x\parallel \sum\limits_{k=n}^{m}
d^{m-k} r^{m-k}s^{k}=\\
&=N (dr)^{m}\parallel \mathcal{A}_{m}^{n}x\parallel
\sum\limits_{k=n}^{m} \left(\dfrac{s}{dr}\right)^{k}\\
&=N (dr)^{m}\parallel \mathcal{A}_{m}^{n}x\parallel.
\end{align*}

{\it Sufficiency.} Using inequality (\ref{t26eqii}) for all
$(m,n,x)\in \Delta\times X$ we have that:
$$d^{m-n}\parallel x\parallel \leq \sum\limits_{k=n}^{m} d^{m-k}\parallel \mathcal{A}_{k}^{n}x\parallel \leq D c^{m}\parallel \mathcal{A}_{m}^{n}x\parallel.$$
 Hence,
 $$\parallel x\parallel \leq D c^{n}
 \left(\dfrac{c}{d}\right)^{(m-n)}\parallel \mathcal{A}_{m}^{n}x\parallel.$$
 If we consider $y=\mathcal{A}_{n}^{p}x$ we obtain
 $$\parallel \mathcal{A}_{n}^{p}y\parallel\leq D\left(\dfrac{c}{d}\right)^{(m-n)}\parallel
 \mathcal{A}_{m}^{p}y\parallel,$$
for all $(m,n,p,y)\in T\times X,$
 which  proves that equation (\ref{A}) is
 p.is.
\end{proof}
\begin{prop}\label{theorem3}
The linear discrete-time system (\ref{A}) is strongly power
instable if and only if there exist some constants $D \geq 1,$ $d
>1$ and $c \geq 1$ with $1\leq 2c < d$ such that:
\begin{equation}\label{t2eqii}
\sum\limits_{k=n}^{m} d^{m-k}\parallel \mathcal{A}_{k}^{n}
x\parallel \leq D c^{m}
\parallel \mathcal{A}_{m}^{n}x\parallel,\;\; \text{for all}\;\; (m,n,x)\in \Delta \times X.
\end{equation}
\end{prop}
\begin{proof}
It is similar with the proof of the Theorem \ref{theorem2}.
\end{proof}
\begin{cor}\label{theorem1}
The linear discrete-time system (\ref{A})  is uniformly power
instable if and only if there exist some constants $D \geq 1$ and
$d
>1$ such that:
\begin{equation}
\sum\limits_{k=n}^{m} d^{m-k}\parallel \mathcal{A}_{k}^{n}
x\parallel \leq D
\parallel \mathcal{A}_{m}^{n}x\parallel,\;\; \text{for all}\;\; (m,n,x)\in \Delta\times X.
\end{equation}
\end{cor}
\begin{rem}
Theorem \ref{theorem2}, Proposition  \ref{theorem3}  and Corollary
\ref{theorem1} are generalizations of Przyluski  type theorems for
the concepts of power instability, strongly power instability and
uniform power instability.
\end{rem}

\end{document}